\documentclass{amsart}
\usepackage{amsmath,amsthm,amssymb}

\title{Canonical immunity and genericity}

\subjclass[2010]{03D28, 68Q30}
\keywords{canonical immunity, $n$-generic, computable Mathias forcing, Schnorr randomness}

\author{Achilles A.~Beros \and Konstantinos A.~Beros}

\address[A. Beros]{Department of Mathematics, University of Hawai'i at M\={a}noa}
\email{beros@math.hawaii.edu}

\address[K. Beros]{Department of Mathematics, Miami University}
\email{berosk@miamioh.edu}

\theoremstyle{plain}
\newtheorem{theorem}{Theorem}[section]
\newtheorem{proposition}[theorem]{Proposition}
\newtheorem{lemma}[theorem]{Lemma}
\newtheorem{corollary}[theorem]{Corollary}
\newtheorem*{theorem*}{Theorem}
\newtheorem*{corollary*}{Corollary}

\newtheorem*{thm::schnorr}{Theorem~\ref{schnorr}}
\newtheorem*{thm::2-generic}{Theorem~\ref{2-generic}}
\newtheorem*{thm::mathias-generic}{Theorem~\ref{mathias-generic}}
\newtheorem*{thm::eff}{Theorem~\ref{eff}}

\theoremstyle{definition}
\newtheorem{definition}[theorem]{Definition}

\theoremstyle{remark}

\newtheorem*{remark}{Remark}

\usepackage{hyperref,color}

\hypersetup{
pdfview=fitH,
pdfpagemode=UseNone,
colorlinks=true,
allcolors=blue,
linktocpage=true
}

\begin{document}

\begin{abstract}Whereas the usual notions of immunity -- e.g., immunity, hyperimmunity, etc. -- are associated with Cohen genericity, {\em canonical immunity}, as introduced in Beros--Khan--Kjos-Hanssen \cite{bkk}, is associated instead with Mathias genericity.  Specifically, every Mathias generic is canonically immune and no Cohen 2-generic computes a canonically immune set.\end{abstract}

\maketitle


\section{Introduction}

When examining questions of algorithmic complexity, there are two principal approaches.  The first of these approaches is based on Baire category; the second on measure.  In each case, there are a variety of notions which formalize the concept of an unremarkable, or typical, subset of $\omega$ -- the set of all natural numbers.  Following a convention established by set theorists, subsets of $\omega$ are referred to as {\em reals}.

According to the category approach, a real is unremarkable if it is generic.  That is, an unremarkable real must be an element of some prescribed collection of dense open subsets of the Cantor space $2^\omega$.  The simplest example is that of a weakly 1-generic real: a real is {\em weakly 1-generic} iff it is a member of every dense computably enumerable (c.e.)~open subset of $2^\omega$.  (An open set is c.e.~if it is the union of a set of basic neighborhoods determined by a c.e.~set of binary strings.)  In a sense, weak 1-genericity is an effective form of Cohen genericity as considered in the theory of forcing.

A number of natural notions arising in computablility theory turn out to be inherent properties of generic reals.  For example, weakly 1-generic reals are {\em hyperimmune}.  That is, the increasing enumeration of a weakly 1-generic subset of $\omega$ is not computably bounded.

The present paper considers the following form of Cohen genericity.

\begin{definition}
A real $R \in 2^\omega$ is {\em $n$-generic} iff, for any $\Sigma^0_n$ set of strings $X \subseteq 2^{<\omega}$, either some initial segment of $R$ is in $X$ or some initial segment of $R$ has no extension in $X$.  
\end{definition}

Turning to the measure theory approach, the associated formulation of ``unremarkable'' is that coming from algorithmic randomness.  In the broadest terms, a real is random if it avoids every member of a specified class of null sets.  For instance, a real is {\em Martin-L\"of random} if it is contained in no null $\Pi^0_2$ class $\bigcap_n U_n$ such that the Lebesgue measure of $U_n$ is bounded by $2^{-n}$.  These null sets are referred to as {\em Martin-L\" of tests}.  It has been argued that the definition of a Martin-L\"of test is too broad since the measures of the sets $U_n$ are potentially only left c.e.~real numbers.  Of interest here is the following weaker notion of randomness.

\begin{definition}
A real is {\em Schnorr random} iff it is a member of no null $\Pi^0_2$ class $\bigcap_n U_n \subseteq 2^\omega$ where the measure of $U_n$ is uniformly computable in $n$ and bounded by $2^{-n}$.
\end{definition}

Intriguingly, there is a form of immunity -- called canonical immunity -- which is very closely associated with Schnorr randomness.  The definition of canonical immunity is due to Bj\o rn Kjos-Hanssen and requires the concept of a canonical numbering.

\begin{definition}
Let $\mathcal P_{\sf fin} (\omega)$ denote the set of finite subsets of $\omega$.  A {\em canonical numbering} is a surjective function $D: \omega \rightarrow \mathcal P_{\sf fin}(\omega)$ such that the predicate
\[
P(x,e) \iff x \in D(e)
\]
and the function
\[
e \mapsto \max (D(e))
\]
are both computable.  Equivalently, if $\mathcal P_{\sf fin} (\omega)$ is identified with the set of canonical codes for finite subsets of $\omega$, a canonical numbering can be regarded as a computable surjection $D : \omega \rightarrow \mathcal P_{\sf fin} (\omega)$.
\end{definition}

\begin{definition}[\cite{bkk}]
An infinite set $R \subseteq \omega$ is {\em canonically immune} iff there is a total computable function $h : \omega \rightarrow \omega$ such that, given any fixed canonical numbering $D : \omega \rightarrow \mathcal P_{\sf fin} (\omega)$, 
\[
D(i) \subseteq R \implies |D(i)| \leq h(i)
\]
for all but finitely many $i$.  Such a witnessing function $h$ is a {\em modulus of immunity} for $R$.
\end{definition}

The first author, together with Mushfeq Khan and Bj\o rn Kjos-Hanssen \cite{bkk}, have shown that every canonically immune set is immune and, moreover, that every Schnorr random is canonically immune with modulus of immunity $i \mapsto i$.  The present paper provides some counterpoints to these observations and shows that the notion of canonical immunity is quite distinct from the standard notions of immunity and genericity in computability theory.  The first main result of this paper shows in a strong way that canonical immunity is not a property of Cohen generic reals.

\begin{thm::2-generic}
If $G \in 2^\omega$ is 2-generic, $G$ computes no canonically immune set.
\end{thm::2-generic}

A corollary is that the class of reals which compute no canonically immune set is comeager.

Theorem~\ref{2-generic} must be contrasted with the fact that, by Theorem~\ref{CI+HI} below, there are reals which are both canonically immune and hyperimmune.  Kurtz \cite{kurtz} has shown that every hyperimmune is of weak 1-generic degree.  Thus, there are weakly 1-generic reals which are Turing equivalent (and hence compute) canonically immune sets -- although the former must not be 2-generic by Theorem~\ref{2-generic}.

There is another form of genericity in computability theory which arises from an effective version of a forcing poset introduced by A.~R.~D.~Mathias \cite{mathias}.  Computable Mathias forcing has been studied in several papers.  Notably, Cholak, Dzhafarov, Hirst and Slaman \cite{cdhs} show that every Mathias generic computes an $n$-generic.  Section~\ref{mathias} below explores the connection between canonical immunity and Mathias genericity.  Whereas $n$-generics (for $n \geq 2$) do not even compute canonically immune sets, Mathias generic sets are always canonically immune.  Section~\ref{preliminaries} contains the necessary definitions relevant to Mathias forcing and Section~\ref{mathias} contains a proof of the following result.

\begin{thm::mathias-generic}
Every Mathias generic is canonically immune.
\end{thm::mathias-generic}

Using further forcing arguments, the last theorem can be used to show the following.

\begin{thm::eff}
There exists a set $R \subseteq \omega$ which is canonically immune and computes no effectively immune set.
\end{thm::eff} 

In light of Theorems~\ref{2-generic} and \ref{mathias-generic}, canonical immunity can be regarded as a form of immunity associated with Mathias genericity rather than Cohen genericity.

The final main result of this paper shows that there are canonically immune sets which are not Schnorr random, demonstrating that the class of canonically immune sets properly contains the class of Schnorr random reals.

\begin{thm::schnorr}
There exists a set $R\subseteq \omega$ which is a canonically immune set and not Schnorr random.
\end{thm::schnorr}

The proof of this result hinges on the observation that Mathias generics cannot be Schnorr random.


\section{Preliminaries and notation}\label{preliminaries}

\subsection{Standard notation}

To a great extent, this paper follows the notation and terminology of Soare \cite{soare}.  For the sake of completeness, the most important points are presented below.

For $e \in \omega$, let $\{ e \} (n)$ denote the result of applying the Turing machine coded by $e$ to input $n$ -- regardless of whether or not this computation terminates.  If $A$ is a subset of $\omega$, a function $\omega \rightarrow \omega$, or a finite string, then $\{ e \}^A (n)$ denotes the result of applying the oracle machine coded by $e$, with oracle $A$, to input $n$.  When a finite string is used as an oracle and the oracle machine makes any queries outside the domain of the string, the computation automatically diverges.  

The notation $W_e$ indicates the domain of the Turing machine coded by $e$.  If $A$ is an oracle, $W^A_e$ denotes the domain of the oracle machine coded by $e$ with oracle $A$.  When computation time is restricted, $W^A_{e,s}$ denotes the domain of the oracle machine coded by $e$ with oracle $A$, when it is only allowed to run for $s$ computation stages.

For sets $A , B \subseteq \omega$, write $A \leq_T B$ if $A$ is Turing reducible to $B$, i.e., there is a code $e$ for an oracle machine such that $\{ e \}^B$ is the characteristic function of $A$.

For finite strings $\alpha , \beta \in 2^{< \omega}$, write $\alpha \preceq \beta$ to indicate that $\alpha$ is an initial segment of $\beta$.  Let $\alpha {}^\smallfrown \beta$ denote the concatenation of $\alpha$ and $\beta$.  If $R \subseteq 2^\omega$ and $\alpha \in 2^{<\omega}$, the notation $\alpha \prec R$ indicates that $\alpha$ is an initial segment of $R$.  For a string $\alpha$, the length of $\alpha$ is indicated by ${\rm length} (\alpha)$.  

If $X \subseteq 2^{<\omega}$ is a set of binary strings and $\sigma \in 2^{<\omega}$,  then $X$ is {\em dense below $\sigma$} iff every $\tau \succeq \sigma$ has an extension in $X$.  If every binary string has an extension in $X$, then $X$ is called {\em dense}.

If $F \subseteq \omega$ is a finite set and $\alpha \in 2^{<\omega}$ is a finite string, $F \subseteq \alpha$ means that $\alpha (n) = 1$ for each $n \in F$, i.e., $F$ is a subset of the finite set of which $\alpha$ is the characteristic function.

For any set $S$, let $|S|$ be the cardinality of $S$.


A function $f : \omega \rightarrow \omega$ is $\Delta^0_2$ or {\em limit computable} iff there is a uniformly computable sequence $(f_s)_{s \in \omega}$ of total functions such that $f$ is the pointwise limit of $(f_s)_{s \in \omega}$.

If $F \subseteq \omega$ is a finite set, the {\em canonical code} for $F$ is the integer $\sum_{n \in F} 2^n$.  The set $\mathcal P_{\sf fin} (\omega)$ of all finite subsets of $\omega$ is identified with the set of all canonical codes.  With this in mind, it is sensible to consider computable functions $D : \omega \rightarrow \mathcal P_{\sf fin} (\omega)$.  The key properties of a computable function $D : \omega \rightarrow \mathcal P_{\sf fin} (\omega)$ are
\begin{enumerate}
	\item the predicate ``$x \in D(i)$'' is computable and
		\item the function $i \mapsto \max (D(i))$ is computable.
\end{enumerate} 

\subsection{The Ellentuck topology and Mathias forcing}

While $n$-genericity derives from a study of the standard topology on $2^\omega$, Mathias genericity is based upon a non-standard topology on the space of infinite subsets of $\omega$.  Let $[\omega]^{\aleph_0}$ denote the set of all infinite subsets of $\omega$, regarded as a $G_\delta$ subspace of $2^\omega$.  Given a finite set $a \subseteq \omega$ and an infinite set $A \subseteq \omega$ with $\max (a) < \min (A)$, let
\[
[a,A] = \{ R \in [\omega]^{\aleph_0} : a \subseteq R \subseteq a \cup A\}.
\] 
The sets $[a,A]$ form the basis of a strong Choquet (hence Baire) topology on $[\omega]^{\aleph_0}$, which is not second countable.  This topology is called the {\em Ellentuck topology}.  The fact that the Ellentuck topology is Baire implies that the countable intersection of dense open sets is nonempty.  Section 19.D of Kechris \cite{kechris} is a clear account of the most important properties of the Ellentuck topology.

\begin{remark}
For $[a,A]$ and $[b,B]$ as above, $[b,B] \subseteq [a,A]$ iff $a \subseteq b$, $b\setminus a \subseteq A$ and $A \supseteq B$.  (See Kechris \cite{kechris}, exercise 19.12.)
\end{remark}

The analogue of the Ellentuck topology in the context of computability theory is the {\em effective Ellentuck topology}, wherein the basic open sets consist only of those $[a , A]$ such that $A$ is an infinite computable subset of $\omega$.  This topology is also strong Choquet and is, in addition, second countable.  The effective Ellentuck topology is the basis of computable Mathias forcing (see Cholak, Dzhafarov, Hirst and Slaman \cite{cdhs}).  The basic open sets in the effective Ellentuck topology are called {\em Mathias conditions}.

\begin{definition}
If $\mathcal X$ is a family of Mathias conditions -- i.e., basic open sets in the effective Ellentuck topology -- and $R \in [\omega]^{\aleph_0}$, then $R$ {\em meets} $\mathcal X$ iff there exists $[a,A] \in \mathcal X$ such that $R \in [a,A]$.  A family $\mathcal X$ of Mathias conditions is {\em dense} iff every Mathias condition contains a member of $\mathcal X$.  A set $R \subseteq \omega$ is {\em Mathias generic} iff $R$ meets every arithmetically definable dense set of Mathias conditions.
\end{definition}

A set $\mathcal X$ of Mathias conditions is dense iff the open set $\bigcup \mathcal X$ is dense in the effective Ellentuck topology on $[\omega]^{\aleph_0}$.  This observation, combined with the fact that the effective Ellentuck topology is Baire, guarantees the existence of Mathias generic reals.

\begin{remark}
If $R \subseteq \omega$ is Mathias generic, then $R$ is infinite.  This follows from the fact that
\[
\mathcal X_n = \{ [a,A] : |a| \geq n\}
\]
is a dense set of conditions.
\end{remark}


\section{Basic results}

The proof of the following theorem shows a method for producing a $\Delta^0_2$ canonically immune set.  An immediate consequence of this is that there are canonically immune sets which do not compute any 2-generic reals.

\begin{theorem}\label{delta two}
There is a $\Delta^0_2$ set $R \subseteq \omega$ which is canonically immune.
\end{theorem}

\begin{proof}
Let $\varphi : \omega^2 \rightarrow \mathcal P_{\sf fin} (\omega)$ be a universal partial computable function and define
\[
D_{e,s} (i) = \begin{cases}
\varphi_e (i) &\mbox{if $\varphi_e (i)$ converge in $s$ or fewer stages},\\
\emptyset &\mbox{otherwise}.
\end{cases}
\]
In particular, given a canonical numbering $D : \omega \rightarrow \mathcal P_{\sf fin} (\omega)$, there is an $e \in \omega$ such that for each $i$,
\[
D(i) = \lim_s D_{e,s} (i).
\]

For $s \in \omega$, inductively pick $x_{n,s} \in \omega$ to be least such that 
\[
x_{0,s} < x_{1,s} < \ldots
\]
and
\[
x_{n,s} \notin \bigcup \{ D_{e,s} (i) : e,i \leq n \mbox{ and } |D_{e,s} (i)| > i\}.
\]
Let $x_n = \lim_s x_{n,s}$.  For each $n \in \omega$, this limit exists because $s$ may be chosen large enough that $D_{e,s} (i)$ has stabilized for all $e,i \leq n$.  Note that 
\[
x_0 < x_1 < \ldots 
\]
and let 
\[
R = \{ x_n : n \in \omega\}.
\]
Observe that $R$ is $\Delta^0_2$.

{\em Claim.}  $R$ is canonically immune.

It suffices to show that, for each canonical numbering $D$, there exists $k \in \omega$ such that $|D(i)| \leq i$ whenever $D(i) \subseteq R$ and $i \geq k$.  Indeed, fix a canonical numbering $D : \omega \rightarrow \mathcal P_{\sf fin} (\omega)$ with $D = \lim_s D_{e,s}$.  Suppose that $D(i) \subseteq R$ for some $i \geq e$.  Let $n$ be least such that
\[
D(i) \subseteq \{ x_0 , \ldots , x_n\}.
\]
Pick $s$ large enough that $D_{j,s} (q)$ (for $j,q \leq i$) and $x_{p,s}$ (for $p \leq n$) have all stabilized.

If $n < i$, then 
\[
|D(i)| \leq n+1 \leq i.
\]
On the other hand, if $n \geq i$, the choice of $x_{n,s}$ and the fact that 
\[
x_{n,s} = x_n \in D(i) = D_{e,s} (i)
\]
guarantee $|D(i)| \leq i$ since $e , i \leq n$.

As $D$ was arbitrary, this shows that $R$ is canonically immune and completes the proof.
\end{proof}

\begin{theorem}\label{bci}
There is a set $R \subseteq \omega$ such that both $R$ and $\omega \setminus R$ are canonically immune, i.e., $R$ is {\em bi-canonically immune}.
\end{theorem}

\begin{proof}
Let $D_0 , D_1 , \ldots$ list all canonical numberings and let $\langle \cdot , \cdot \rangle$ be a fixed computable pairing function.  Let $f : \omega \rightarrow \omega$ be a computable function having the property that
\[
\langle e , i \rangle \leq f(i)
\]
for all $i \geq e$, e.g., $f(i) = \max_{e \leq i} \langle e , i \rangle$.  For convenience, denote by $I_p$ the set $\{ 2p , 2p+1 \}$ for each $p \in \omega$.  The construction of $R$ proceeds in stages and produces finite sets $R_s , Q_s , F_s  \subseteq \omega$ such that
\begin{itemize}
	\item $|F_s| \leq 2s$,
	\item $R_s \cap Q_s = \emptyset$ and 
	\item $R_s \cup Q_s = \bigcup_{p \in F_s} I_p$.
\end{itemize}

To begin with, define
\[
R_0 = Q_0 = F_0 = \emptyset.
\]
Suppose now, at stage $s = \langle e , i \rangle$, the finite sets $R_s$, $Q_s$ and $F_s$ are given with the above properties.  

{\em Case 1.}  If either $i < e$ or $|D_e (i)| \leq 4f(i) + 3$, let $R_{s+1} = R_s$, $Q_{s+1} = Q_s$, $F_{s+1} = F_s$ and end the stage. 

{\em Case 2.}  Suppose $i \geq e$ and $|D_e(i)| > 4f(i) + 3$.  Because $|F_s|  \leq 2s$, the finite union $\bigcup_{p \in F_s} I_p$ has cardinality at most $4s = 4\langle e , i \rangle \leq 4 f(i)$.  Thus, there are distinct $p_s,q_s \notin F_s $ with 
\[
D_e (i) \cap I_{p_s} \neq \emptyset \mbox{ and } D_e (i) \cap I_{q_s} \neq \emptyset.
\]
Let $x , y , z , w \in \omega$ be such that $I_{p_s} = \{ x , y \}$, $I_{q_s} = \{ z , w \}$, $x \in D_e (i) \cap I_{p_s}$ and $z \in D_e (i) \cap I_{q_s}$.  Now define 
\begin{itemize}
	\item $R_{s+1} = R_s \cup \{ y , z\}$,
	\item $Q_{s+1} = Q_s \cup \{ x , w \}$ and
	\item $F_{s+1} = F_s \cup \{ p_s , q_s \}$.
\end{itemize}
This completes case 2 of stage $s$.

To finish the construction, let $F = \bigcup_s F_s$,
\[
R = \{ 2p : p \notin F\} \cup \bigcup_s R_s 
\]
and
\[
Q = \{ 2p+1 : p \notin F\} \cup \bigcup_s Q_s.
\]
Notice that $R$ and $Q$ are both infinite and $R = \omega \setminus Q$.  The construction is symmetric in $R$ and $Q$.  Hence, the following claim suffices to complete the proof.

{\em Claim.}  $R$ is canonically immune with modulus of immunity $i \mapsto 4f(i) + 3$.

Indeed, fix a canonical numbering $D_e : \omega \rightarrow \mathcal P_{\sf fin} (\omega)$ and suppose
\[
|D_e(i)| > 4f(i) + 3
\]
for some $i \geq e$.  By adding $p_s$ to $F_s$, the strategy at stage $s = \langle e , i \rangle$ renders $I_{p_s}$ unavailable for use at later stages and guarantees that $R_t$ does not contain $D_e (i)$ for any $t \geq s$.  Specifically, the element $x \in D_e (i) \cap I_{p_s}$ is permanently withheld from $\bigcup_s R_s$.  Since $\{ 2p : p \notin F \}$ is disjoint from $\bigcup_s R_s$, the larger set $R$ must also not contain $D_e (i)$.
\end{proof}

\begin{theorem}
The canonically immune degrees are cofinal in the Turing degrees.
\end{theorem}

\begin{proof}
The objective is to construct a canonically immune set $R$ such that $A \leq_T R$ for any fixed $A \subseteq \omega$.  Let $D_0 , D_1 , \ldots$ list all canonical numberings.  Choose natural numbers 
\[
p_0 < p_1 < \ldots
\]
such that 
\[
2p_n , 2p_n +1 \notin \bigcup \{ D_e (i) : e,i \leq n \mbox{ and } |D_e(i)| > i\}.
\]
Let
\[
Q = \{ 2p_n : n \in \omega \} \cup \{ 2p_n + 1 : n \in \omega\}.
\]
To see that $Q$ is canonically immune with modulus of immunity $i \mapsto 2i+1$, suppose $D_e (i) \subseteq Q$ for some $i \geq e$.  Pick the least $n$ such that
\[
D_e (i) \subseteq \{ 2p_j : j \leq n \} \cup \{ 2p_j + 1 : j \leq n\}.
\]
In particular, either $2p_n \in D_e (i)$ or $2p_n + 1 \in D_e (i)$.  If $n < i$, then
\[
|D_e (i)| \leq 2n + 2 \leq 2i + 1.
\]
On the other hand, if $n \geq i$, the choice of $p_n$ guarantees that $|D_e (i)| \leq i$.

Now suppose that $A \subseteq \omega$ is any fixed set.  Let 
\[
R = \{ 2p_n : n \in A\} \cup \{ 2p_n + 1 : n \notin A\}.
\]
As an infinite subset of the canonically immune set $Q$, $R$ must itself be canonically immune.  Also $A \leq_T R$ since, for each $n \in \omega$,
\[
n \in A \iff \mbox{the $n$th element of $R$ is even.}
\]
This completes the proof.
\end{proof}


\section{Cohen generic reals}\label{generic}

Whereas Cohen genericity (especially weak 1-genericity) is closely related to immunity and hyperimmunity, the following initial observation already shows that there is not as strong a connection between Cohen genericity and canonical immunity.

\begin{theorem}
No weak 1-generic is canonically immune.
\end{theorem}

\begin{proof}
Suppose that $G \in 2^\omega$ is weakly 1-generic.  Fix a computable function $f: \omega \rightarrow \omega$ and a canonical numbering $D : \omega \rightarrow \mathcal P_{\sf fin} (\omega)$ such that, for infinitely many $i$,
\begin{itemize}
	\item $\min (D(i)) > i$ and
	\item $|D(i)| > f(i)$.
\end{itemize}
Let
\[
X_n = \{ \sigma \in 2^{< \omega} : (\exists i \geq n)(D(i) \subseteq \sigma \mbox{ and } |D(i)| > f(i) )\}.
\]
Each $X_n$ is c.e.~and dense.  Thus, $G$ has an initial segment in each $X_n$.  In particular, there are infinitely many $i$ such that $D(i) \subseteq G$ and $|D(i)| > f(i)$.  It follows that $f$ is not a modulus of immunity for $G$.  As $f$ was arbitrary, $G$ cannot be canonically immune.
\end{proof}

As a counterpoint, there are weakly 1-generic reals which are Turing equivalent to canonically immune sets.  This is a consequence of the following theorem, along with the fact, due to Kurtz \cite{kurtz}, that every hyperimmune is Turing equivalent to a weak 1-generic.

\begin{theorem}\label{CI+HI}
There is a canonically immune set $R$ which is also hyperimmune.
\end{theorem}

\begin{proof}
Let $D_0 , D_1 , \ldots$ be a list of all canonical numberings and $f_0 , f_1 , \ldots$ a list of all total computable functions.  For each $s \in \omega$, consider the finite set 
\[
F_s = \bigcup \{ D_e (i) : i,e \leq s \mbox{ and } |D_e(i)| > i\}.
\]
Inductively pick $x_s \notin F_s$ large enough that
\[
x_s > x_{s-1} \mbox{ and } x_s > f_s (s).
\]
Let $R = \{ x_s : s \in \omega\}$.

{\em Claim.}  $R$ is canonically immune with modulus of immunity $i \mapsto i$.

Fix a canonical numbering $D = D_e$ and suppose $i \geq e$ is such that 
\[
 D_e(i) \subseteq R  \mbox{ and } |D_e (i) | > i.
\]
By the choice of $x_s$, 
\[
D_e (i) \subseteq \{ x_0 , \ldots , x_{i - 1}\}.
\]
In particular, $|D_e (i)| \leq i$, which is a contradiction.  Thus, for all $i \geq e$, if $D_e(i) \subseteq R$, then $|D_e(i)| \leq i$.

{\em Claim.}  $R$ is hyperimmune.

To see this, note that $x_0 < x_1 < \ldots$ is the increasing enumeration of $R$ and, since each $x_s$ is greater than $f_s (s)$, there is no computable bound for the function
\[
s \mapsto x_s.
\]
This completes the proof.
\end{proof}

On the other hand, a canonically immune set need not be hyperimmune.

\begin{theorem}\label{cinothi}
There exists a set $R \subseteq \omega$ which is canonically immune and such that neither $R$ nor $\omega \setminus R$ are hyperimmune.
\end{theorem}

\begin{proof}
Let $D_0 , D_1 , \ldots$ list all canonical numberings and let $\langle \cdot , \cdot \rangle$ be a fixed computable pairing function.  As in the proof of Theorem~\ref{bci}, choose a computable function $f: \omega \rightarrow \omega$ such that, for all $e$ and $i \geq e$,
\[
\langle e,i \rangle \leq f(i),
\]
and let $I_p = \{ 2p , 2p+1\}$ for $p \in \omega$.

The proof proceeds inductively and produces finite sets $R_s , F_s \subseteq \omega$ such that
\begin{itemize}
	\item $|F_s| \leq s$ and
	\item $R_s \subseteq \bigcup_{p \in F_s} I_p$.
\end{itemize}
To begin the induction, let $R_0 = F_0 = \emptyset$.  At stage $s = \langle e,i \rangle$, suppose that $R_s$ and $F_s$ are given with the above properties.  There are two cases.   

{\em Case 1.}  If $i < e$ or $|D_e (i)| \leq 2 f(i)$, let $R_{s+1} = R_s$ and $F_{s+1} = F_s$.

{\em Case 2.}  Suppose $i \geq e$ and $|D_e (i)| > 2 f(i)$.  Also, $|D_e(i)| > 2s$ since $f(i) \geq \langle e,i \rangle = s$ by the choice of $f$.  It follows that
\[
D_e(i) \setminus \bigcup_{p \in F_s} I_p \neq \emptyset
\]
because $|F_s| \leq s$ and each $I_p$ has cardinality 2.  Therefore, pick $p_s \notin F_s$ and $x_s \in I_{p_s}$ such that
\[
x_s \in D_e(i) \cap I_{p_s}.
\]
Let 
\[
R_{s+1} = R_s \cup (I_{p_s} \setminus \{ x_s \})
\]
and
\[
F_{s+1} = F_s \cup \{ p_s \}.
\]

Now let $F = \bigcup_s F_s$ and 
\[
R = \bigcup_s R_s \cup \{ 2p : p \notin F\}.
\]

{\em Claim.}  $R$ is canonically immune with modulus of immunity $i \mapsto 2 f(i)$.

Suppose that $D_e : \omega \rightarrow \mathcal P_{\sf fin} (\omega)$ is a canonical numbering and $|D_e (i)| > 2f(i)$ for some $i \geq e$.  As in the proof of Theorem~\ref{bci}, the strategy at stage $s = \langle e  , i \rangle$ ensures that the element $x_s \in D_e (i)$ is withheld from $R$.

{\em Claim.}  $R$ and $\omega \setminus R$ are not hyperimmune.

Indeed, for each $p \in \omega$, 
\[
| R \cap I_p | = |(\omega \setminus R) \cap I_p| = 1.
\]
Given that $I_p = \{ 2p , 2p+1\}$, the $k$th elements of $R$ and $\omega \setminus R$ must both be no more that $2k$.  In other words, the increasing enumerations of $R$ and $\omega \setminus R$ are bounded by the computable function $k \mapsto 2k$. 
\end{proof}

\begin{remark}
Theorem~\ref{cinothi} could also be obtained by constructing a Schnorr random of density $\frac{1}{2}$ and then using the fact that every Schnorr random is canonically immune along with the observation that a set of positive density cannot be hyperimmune.
\end{remark}

The main result proved in this section (Theorem~\ref{2-generic}) states that no 2-generic real computes a canonically immune set.  Although the next theorem is a consequence of Theorem~\ref{2-generic}, its proof illustrates the method used to diagonalize against canonically immune sets in the proof of Theorem~\ref{2-generic}.

\begin{theorem}
There is a set $R \subseteq \omega$ which is hyperimmune and not canonically immune.
\end{theorem}

\begin{proof}
For each total computable function $f : \omega \rightarrow \omega$, let $H_f : \omega \rightarrow \mathcal P_{\mathsf{fin}} (\omega)$ be a computable function such that 
\begin{itemize}
	 \item $|H_f (n)| > f(2n)$,
	 \item $\min \left(H_f (n)\right) \geq n$ and
	 \item the $H_f (n)$ are pairwise disjoint sets.
\end{itemize}
Let $f_0 , f_1 , \ldots$ list all total computable functions and let $\langle \cdot , \cdot \rangle$ be a fixed computable pairing function.  Define $n_p$ inductively as follows: if $0 = \langle a , b \rangle$, let $n_0 > f_a (0)$.  Given $p = \langle i , k \rangle$ and $n_0 , \ldots , n_{p-1}$, let
\[
s = \sum_{\langle j,q \rangle < p} \left| H_{f_j} (n_{\langle j , q \rangle}) \right|
\]
and choose $n_p \in \omega$ such that
\begin{itemize}
	\item $n_p > f_i (s+1)$ and
	\item $H_{f_i} (n_p)$ is disjoint from all $H_{f_j} (n_{\langle j,q \rangle})$ for $\langle j,q \rangle < p$.
\end{itemize}
It possible to choose such an $n_p$ because the sets $H_{f_i} (n)$ are pairwise disjoint and nonempty for each fixed $f_i$.  Let
\[
R = \bigcup_{i,k \in \omega} H_{f_i} (n_{\langle i,k \rangle}).
\]
Noting that $R$ is infinite, the following two claims suffice to complete the proof.

{\em Claim.}  $R$ is hyperimmune.

It suffices to show that the increasing enumeration of $R$ is not eventually bounded by any of the $f_i$.  Indeed, fix $k \in \omega$ and let
\[
s = \sum_{\langle j,q \rangle < \langle i , k \rangle} \left| H_{f_j} (n_{\langle j,q \rangle}) \right|.
\]
The minimum element of $H_{f_i} (n_{\langle i , k \rangle})$ is therefore the $(s+1)$st element of $R$.  By the choice $n_{\langle i , k \rangle}$ and the definition of $H_{f_i}$,
\[
f_i(s+1) < n_{\langle i , k \rangle} \leq \min \left( H_{f_i} (n_{\langle i,k \rangle}) \right).
\]
As $k$ was arbitrary, the increasing enumeration of $R$ must infinitely often exceed $f_i$.

{\em Claim.}  $R$ is not canonically immune.

Fix one of the $f_i$.  To see that $f_i$ is not a modulus of immunity for $R$, let $D : \omega \rightarrow \mathcal P_{\mathsf{fin}} (\omega)$ be a canonical numbering with the property that
\[
D(2n) = H_{f_i} (n)
\]
for each $n \in \omega$.  For $k \in \omega$, 
\[
D(2 n_{\langle i, k \rangle}) \subseteq R
\]
and, by the definition of $H_{f_i}$,
\[
\left| D(2 n_{\langle i,k \rangle}) \right| > f_i (2 n_{\langle i,k \rangle}).
\]
This completes the proof.
\end{proof}

As discussed in the introduction, no 2-generic bounds a canonically immune real.

\begin{theorem}\label{2-generic}
If $G \in 2^\omega$ is 2-generic, $G$ computes no canonically immune real.
\end{theorem}

\begin{proof}
Fix a 2-generic real $G \in 2^\omega$.  It suffices to show that $W^G_e$ is not canonically immune for any $e \in \omega$.  To this end, fix $e \in \omega$ and a computable function $f: \omega \rightarrow \omega$.  The objective of the proof is to show that $f$ is not a modulus of immunity for $W^G_e$.  Consider first the $\Pi^0_1$ set of strings
\[
X = \{ \sigma \in 2^{< \omega} : (\forall \tau \succeq \sigma) (W^\tau_{e , {\rm length} (\tau)} =  W^\sigma_{e , {\rm length} (\sigma)})\}.
\]

As $G$ is 2-generic, there is either an initial segment of $G$ in $X$, or some $\sigma \prec G$ such that no extension of $\sigma$ is in $X$.  In the first case, $W^G_e$ is finite and consequently not canonically immune.  In the second case, suppose $\sigma \prec G$ is such that no extension of $\sigma$ is in $X$.  In particular, for each $\tau \succeq \sigma$, there exists $\tau' \succeq \tau$ with 
\[
W^{\tau'}_{e , {\rm length} (\tau')} \setminus W^\tau_{e , {\rm length} (\tau)} \neq \emptyset.
\]
By induction, it follows that, for each $\tau \succeq \sigma$ and $n \in \omega$, there is a $\rho \succeq \tau$ with 
\[
\left| W^\rho_e \right| \geq n.
\]
Fix a computable enumeration $\alpha_0 , \alpha_1 , \ldots$ of $2^{<\omega}$ and a computable pairing function $\langle \cdot , \cdot \rangle$.  Let $\beta : \omega^2 \rightarrow 2^{<\omega}$ be a computable function such that, for each $i , n \in \omega$,
\[
|W^{\sigma {}^\smallfrown \alpha_i {}^\smallfrown \beta( i , n )}_e| > f(2 \langle i , n \rangle).
\]
The function $\beta$ is defined for every pair $i,n$ since each $\tau \succeq \sigma$ has an extension $\rho$ with $|W^\rho_e| > f(2\langle i , n \rangle)$ by remarks above.  Let $H : \omega^2 \rightarrow \mathcal P_{\sf fin} (\omega)$ be a computable function such that, for $i,n \in \omega$,
\begin{itemize}
	\item $|H(i,n)| > f( 2 \langle i , n \rangle)$ and 
	\item $H( i , n ) \subseteq W^{\sigma {}^\smallfrown \alpha_i {}^\smallfrown \beta( i , n )}_e$.
\end{itemize}
For instance, $H$ could output a canonical index for the finite set consisting of the first $f(2\langle i , n \rangle) + 1 $ elements enumerated into $W^{\sigma {}^\smallfrown \alpha_i {}^\smallfrown \beta( i , n )}_e$.

Consider now the $\Sigma^0_1$ set of strings
\[
Y_n = \{ \sigma {}^\smallfrown \alpha_i {}^\smallfrown \beta ( i , n ) : i , n \in \omega\}. 
\]
Observe that each $Y_n$ is dense below $\sigma$.  Hence, each $Y_n$ must contain an initial segment of $G$ since $G$ is 2-generic and $\sigma \prec G$.  For each pair $i , n \in \omega$ with $\sigma {}^\smallfrown \alpha_i {}^\smallfrown \beta (i , n) \prec G$,
\[
H ( i , n ) \subseteq W^{\sigma {}^\smallfrown \alpha_i {}^\smallfrown \beta( i , n )}_e \subseteq W^G_e,
\]
by the choice of $H$.  Therefore, let $D : \omega \rightarrow \mathcal P_{\sf fin} (\omega)$ be any canonical numbering such that
\[
D( 2 \langle i , n \rangle ) = H( i , n ).
\]
Whenever $\sigma {}^\smallfrown \alpha_i {}^\smallfrown \beta (i , n) \prec G$,
\[
D(2 \langle i , n \rangle) \subseteq W^G_e.
\]
In particular, there are infinitely many $k$ such that $D(k) \subseteq W^G_e$ and $|D(k)| > f(k)$.  As $f$ was arbitrary, it follows that $W^G_e$ is not canonically immune.
\end{proof}

\begin{remark}
The proof of Theorem~\ref{2-generic} actually establishes a stronger result than necessary: no canonically immune real is $\Sigma^0_1$ (i.e., c.e.) in a 2-generic.
\end{remark}

Because the  set of 2-generic reals is comeager, Theorem~\ref{2-generic} yields the following corollary.

\begin{corollary}
The set of reals which bound no canonically immune set is comeager.
\end{corollary}


\section{Mathias generic reals}\label{mathias}

In this section, if $a \subseteq \omega$ is a finite set, $\chi_a \in 2^{<\omega}$ denotes the binary string of length $\max(a) + 1$ such that
\[
\chi_a (n) = 1 \iff n \in a.
\]

As described in the introduction, there is a relationship between canonical immunity and Mathias genericity.

\begin{theorem}\label{mathias-generic}
Every Mathias generic is canonically immune.
\end{theorem}

\begin{proof}
For each canonical numbering $D : \omega \rightarrow \mathcal P_{\sf fin} (\omega)$, define a family $\mathcal X_D$ of Mathias conditions by
\[
\mathcal X_D = \{ [a,A] : (\forall i \geq |a| ) (D(i) \subseteq a \cup A \implies |D(i)| \leq i)\}.
\]

{\em Claim.}  For each canonical numbering $D$, the set $\mathcal X_D$ of Mathias conditions is dense.

Indeed, fix a Mathias condition $[a,A]$.  Let $n \mapsto x_n$ be the increasing enumeration of the infinite computable set $A$.  Inductively choose 
\[
n_{|a|+1} < n_{|a|+ 2} < \ldots,
\]
where each $n_i$ is least such that
\[
x_{n_i} \notin \bigcup_{|a| \leq j < i} D(j)
\]
for $i > |a|$.  Let 
\[
B = \{ x_{n_i} : i > |a|\}
\]
and observe that $B$ is an infinite computable subset of $A$.  To see that $[a,B] \in \mathcal X_D$, suppose $D(i) \subseteq a \cup B$ for some $i \geq |a|$.  By the choice of $n_i$,
\[
D(i) \subseteq a \cup \{ x_{n_{|a|+1}} , \ldots , x_{n_i}\}.
\]
In particular, $|D(i)| \leq |a| + (i - |a|) = i$ and hence $[a,B] \in \mathcal X_D$.  Since $B \subseteq A$, it follows that $[a,A] \supseteq [a,B]$.  This shows that $\mathcal X_D$ is dense.

{\em Claim.}  If $R$ meets each $\mathcal X_D$, then $R$ is canonically immune.

Fix a canonical numbering $D$ and let $[a,A]$ be such that 
\[
R \in [a,A] \in \mathcal X_D.
\]
If $i \geq |a|$ and $D(i) \subseteq R$, then also $D(i) \subseteq a \cup A$ and hence, by the definition of $\mathcal X_D$, 
\[
|D(i)| \leq i.
\]
As $D$ was arbitrary, $R$ must be canonically immune with modulus of immunity $i \mapsto i$.  

It now follows that any Mathias generic real is canonically immune.
\end{proof}

\begin{remark}
When working with computable Mathias forcing, the typical approach is to identify each Mathias condition $[a,A]$ with a pair $(x,e)$ where $x$ is a canonical code for the finite set $a$ and $e$ is a Turing machine code for the characteristic function of $A$.  Thus, the statement that $[a,A]$ is a Mathias condition is equivalent to the $\Pi^0_2$ statement
\begin{enumerate}
	\item $\{ e \}$ is total,
	\item $(\forall n) (\{ e \} (n) \in \{ 0 , 1 \})$,
	\item $(\exists^\infty n ) (\{ e \} (n) = 1)$ (i.e., $A$ is infinite) and
	\item $\max (a) < \min (A)$.
\end{enumerate}
Examining the proof of Theorem~\ref{mathias-generic} reveals that the dense sets $\mathcal X_D$ required to show that Mathias generics are canonically immune are $\Pi^0_1$ definable sets of Mathias conditions.  Thus, the sets of codes for these conditions are $\Pi^0_2$ and hence $\Sigma^0_3$.  In particular, only Mathias 3-genericity is required to obtain canonical immunity.  Refer to \cite{cdhs} for the definition of Mathias $n$-genericity.
\end{remark}

In view of the remarks above, Theorem~\ref{2-generic}, along with the result of Theorem~\ref{mathias-generic}, gives an alternative proof of a result from \cite{cdhs}.

\begin{corollary}[Cholak-Dzhafarov-Hirst-Slaman~\cite{cdhs}]
If $G$ is Cohen 2-generic, then $G$ bounds no Mathias 3-generic.
\end{corollary}

Binns, Kjos-Hanssen, Lerman and Solomon \cite[Corollary 6.7]{Binns}, have shown that every Mathias 3-generic is {\em high}, i.e., if $G$ is Mathias 3-generic, $G' \geq_T \emptyset''$.  On the other hand, Kurtz showed in his Ph.D.~thesis that the set of reals which are not high has Lebesgue measure 1 in $2^\omega$.  In particular, there are Schnorr random reals -- hence, canonically immune reals -- which are not high since the class of Schnorr random reals also has measure 1.  It follows that there is a measure 1 set of canonically immune reals which do not compute Mathias 3-generics.  Therefore, by asserting that no 2-generic computes a canonically immune set, Theorem~\ref{2-generic} above is strictly stronger than the Cholak-Dzhafarov-Hirst-Slaman result.

The last topic of this section is the relationship between canonical and effective immunity.  The following proposition is a consequence of existing results.

\begin{proposition}
Every effectively immune set computes a canonically immune set.
\end{proposition}  

\begin{proof}
Every effectively immune set computes a fixed point free function and hence a DNC.  By  unpublished work of Noam Greenberg and Joe Miller, every DNC computes an infinite subset of a Martin-L\" of random and hence computes an infinite subset of a Schnorr random.  The proposition now follows by Beros--Khan--Kjos-Hanssen \cite[Corollary 5.6]{bkk}, which states that every infinite subset of a Schnorr random computes a canonically immune set.
\end{proof}

\begin{theorem}\label{eff}
There exists a set $R \subseteq \omega$ which is canonically immune and computes no effectively immune set.
\end{theorem}

The following lemma is a typical application of forcing methods in computability theory.  Combining it with the assertion (Theorem~\ref{mathias-generic}) that every Mathias generic is canonically immune yields the theorem above.

\begin{lemma}\label{computes}
If $R \in [\omega]^{\aleph_0}$ is Mathias generic, then $R$ computes no effectively immune set.
\end{lemma}

\begin{proof}
Given $e \in \omega$ and a computable function $h : \omega \rightarrow \omega$, let $\mathcal D_{e , h}$ be the set of Mathias conditions $[a,A]$ such that either
\begin{enumerate}
\item $(\exists i \in \omega) (W_i \subseteq W^{\chi_a}_e \mbox{ and } |W_i| > h(i))$ or 
\item $(\exists n \in \omega) (\forall b \in \mathcal P_{\sf fin} (\omega)) ( a \subseteq b \subseteq a \cup A \implies |W^{\chi_b}_e| \leq n)$.
\end{enumerate}
The sets $\mathcal D_{ e , h }$ are all arithmetical families of Mathias conditions.

{\em Claim.}  Each $\mathcal D_{e , h}$ is dense.

To this end, fix a Mathias condition $[a,A]$.  Suppose first that there exists $n \in \omega$ such that, for every $R \in [a,A]$,
\[
\left|W^R_e\right| \leq n.
\]
In this case, $[a,A] \in \mathcal D_{ e , h }$ since $[a,A]$ satisfies condition (2) above.  

Suppose now that, given any $n \in \omega$, there exists $R \in [a,A]$ such that $\left|W^R_e\right| > n$.  In particular, for every $n \in \omega$, there is a finite set $b$ with $a \subseteq b \subseteq a \cup A$ and
\[
|W^{\chi_b}_e| > n.
\]
Therefore, let $g : \omega \rightarrow \omega$ and $\rho : \omega \rightarrow \mathcal P_{\sf fin} (A)$ be computable functions such that 
\[
W_{g(i)} \subseteq W^{\chi_{a \cup \rho(i)}}_e
\]
and
\[
|W_{g(i)}| > h(i).
\]
Using the Recursion Theorem, let $j$ be a fixed point for the function $g$, i.e., $W_{g(j)} = W_j$.  Set $b = a \cup \rho(j)$ and take 
\[
B = A \cap [\max (\rho(j)) + 1 , \infty).
\]
It follows that $[b,B] \subseteq [a,A]$ and $[b,B]$ satisfies condition (1) above, i.e., $[b,B] \in \mathcal D_{ e , h }$.  This establishes that each $\mathcal D_{ e , h }$ is dense.

Let $R$ be Mathias generic.  Given $e \in \omega$ and a computable function $h : \omega \rightarrow \omega$, let $[a,A] \in \mathcal D_{ e , h }$ with $R \in [a,A]$.  If $[a,A]$ satisfies condition (1) in the definition of $\mathcal D_{ e , h }$, then $W^R_e$ is not effectively immune via $h$ since there exists $i \in \omega$
\[
W_i \subseteq W^{\chi_a}_e \subseteq W^R_e
\]
with $|W_i| > h(i)$.  On the other hand, if $[a,A]$ satisfies condition (2), $W^R_e$ is finite and again not effectively immune.

It now follows that $W^R_e$ is not effectively immune for any $e \in \omega$.  In particular, $R$ computes no effectively immune subset of $\omega$.
\end{proof}

\begin{remark}
On the other hand, every canonically immune set contains an effectively immune set.  In particular, there are reals which are both canonically immune and effectively immune.  Such reals are not bounded by any Mathias generic.

To construct an effectively immune subset of a canonically immune set $R = \{ r_0 < r_1 < \ldots\}$, run the standard construction of an effectively immune set inside $R$:  at stage $s$, pick the least $e \leq s$ such that
\[
X_s = W_e \cap \{ r_{2e} , r_{2e + 1} , \ldots \} \neq \emptyset
\] 
and remove $y_s = \min (X_s)$ from $R$.  The resulting set 
\[
Q = R \setminus \{ y_s : s \in \omega\}
\]
is still infinite since
\[
\left| Q \cap \{ r_0 , \ldots , r_n\}\right| \geq n/2
\]
for each $n \in \omega$.  As an infinite subset of a canonically immune set, $Q$ is canonically immune.  Furthermore, $Q$ is effectively immune because any $W_e  \subseteq Q$ is a subset of 
\[
\{ r_0 , r_1 , \ldots , r_{2e-1}\}
\]
and must therefore have cardinality at most $2e$.
\end{remark}


\section{Canonical immunity vs.~Schnorr randomness}\label{schnorr-randomness}

The final result of this paper is an application of the fact that every Mathias generic is canonically immune.

\begin{theorem}\label{schnorr}
There exists a set $R\subseteq \omega$ which is a canonically immune set and not Schnorr random.
\end{theorem}

\begin{proof}
Because every Mathias generic is canonically immune, it suffices to show that there are Mathias generics which are not Schnorr random.  In fact, it turns out that every Mathias generic is not Schnorr random.  

Let $F_1 , F_2 , \ldots \subseteq \omega$ be pairwise disjoint consecutive intervals with $|F_i| = i$.  For each $n \in \omega$, define the open set
\[
U_n = \{ R \in 2^\omega : (\exists i > n) (R \cap F_i = \emptyset)\}.
\]
Notice that the Lebesgue measure of $U_n$ is exactly $2^{-n}$.  In particular, the $\Pi^0_2$ class
\[
\bigcap_n U_n = \{ R \in 2^\omega : (\exists^\infty i) (R \cap F_i = \emptyset)\}
\]
is a Schnorr test.

{\em Claim.}  There is a dense $\Sigma^0_2$-definable set $\mathcal X$ of Mathias conditions such that $\bigcup \mathcal X \subseteq \bigcap_n U_n$.

To verify this claim, let
\[
\mathcal X = \{ [a,A] : (\exists^\infty i) (A \cap F_i = \emptyset)\}
\]
and observe that $\bigcup \mathcal X \subseteq \bigcap_n U_n$.  To see that $\mathcal X$ is dense, fix any Mathias condition $[a,A]$ with 
\[
A = \{ x_0 < x_1 < \ldots \}.
\]
Inductively choose $x_{m_0} < x_{m_1} < \ldots$ and $i_0 < i_1 < \ldots$ such that the map $p \mapsto x_{m_p}$ is computable and, for all $p \in \omega$,
\[
\max (F_{i_p}) < x_{m_p} < \min (F_{i_{p+1}}).
\]
This is always possible because $A$ is an infinite computable set.  Define a computable subset of $A$ by letting
\[
B = \{ x_{m_p} : p \in \omega\}
\]
and note that there are infinitely many $i$ with $B \cap F_i = \emptyset$, i.e., $[a,B] \in \mathcal X$.  Also, $[a,B] \subseteq [a,A]$ and, as $[a,A]$ was arbitrary, this shows that $\mathcal X$ is dense and establishes the claim.

Any Mathias generic real must meet every dense arithmetical set of conditions.  In particular, a Mathias generic $G$ must be a member of $\bigcup \mathcal X \subseteq \bigcap_n U_n$.  Hence, $G$ is not Schnorr random.
\end{proof}

\begin{remark}
As noted above, Binns, Kjos-Hanssen, Lerman and Solomon \cite[Corollary 6.7]{Binns} showed that every Mathias generic is high.  By Nies--Stephan--Terwijn \cite[Theorem 4.2]{nst}, every high set is Turing equivalent to a Schnorr random.  Thus, although no Mathias generic is Schnorr random, every Mathias generic is of Schnorr random degree.
\end{remark}


\end{document}